\documentclass[12pt]{article}
\usepackage{amsmath,amsfonts,amssymb,amsthm}
\usepackage{xcolor}
\usepackage{enumitem}
\usepackage[colorlinks=true,linkcolor=black,citecolor=black,urlcolor=black,bookmarks,breaklinks=true]{hyperref}
\setitemize{itemsep=0.1pt}
\setenumerate{itemsep=0.1pt}

\usepackage{mathtools}

\newtheorem{proposition}{Proposition}[section]
\newtheorem{theorem}[proposition]{Theorem}

\newtheorem{remark}[proposition]{Remark}

\newtheorem{algorithm}[proposition]{Algorithm}
\newcommand{\nc}{\newcommand}
\nc{\I}{{\mathbf 1}}

\nc{\bN}{{\mathbf N}}
\nc{\bM}{{\mathbf M}}
\nc{\cB}{{\mathcal B}}
\nc{\cL}{{\mathcal L}}
\nc{\R}{{\mathbb R}}
\nc{\N}{{\mathbb N}}
\nc{\Z}{{\mathbb Z}}
\nc{\bF}{{\mathbf F}}

\DeclareMathOperator{\aff}{aff}
\nc{\BP}{\mathbb{P}}
\nc{\BE}{\mathbb{E}}
\nc{\BQ}{\mathbb{Q}}

\nc{\BH}{{\mathbb H}}
\nc{\BS}{{\mathbb S}}

\DeclareMathOperator{\BV}{{\mathbb Var}}
\DeclareMathOperator{\BC}{{\mathbb Cov}}

\numberwithin{equation}{section}

\begin{document} 

\renewcommand{\thefootnote}{\fnsymbol{footnote}}
\author{M.A. Klatt\footnotemark[1]\, and G. Last\footnotemark[2]}
\footnotetext[1]{mklatt@princeton.edu, Department of Physics, Princeton University, Princeton, NJ 08544, USA }
\footnotetext[2]{guenter.last@kit.edu, Karlsruhe Institute of
  Technology, Institute for Stochastics, 76131 Karlsruhe, Germany. }

\title{On strongly rigid hyperfluctuating random measures} 
\date{\today}
\maketitle

\begin{abstract} 
\noindent 
In contrast to previous belief, we provide examples of stationary ergodic
random measures that are both hyperfluctuating and strongly rigid.
Therefore, we study hyperplane intersection processes (HIPs) that are
formed by the vertices of Poisson hyperplane tessellations.
These HIPs are known to be hyperfluctuating, that
is, the variance of the number of points in a bounded observation window
grows faster than the size of the window.
Here we show that the HIPs exhibit a particularly strong rigidity
property.
For any bounded Borel set $B$, an exponentially small (bounded) stopping
set suffices to reconstruct the position of all points in $B$ and, in
fact,
all hyperplanes intersecting $B$.
Therefore, also the random measures supported by the hyperplane
intersections of arbitrary (but fixed) dimension, are hyperfluctuating.
Our examples aid the search for relations between correlations, density
fluctuations, and rigidity properties.
\end{abstract}

tessellations, hyperplane intersection processes

\vspace{0.1cm}

\section{Introduction}\label{sintro}

Let $\Phi$ be a {\em random measure} on the $d$-dimensional
Euclidean space $\R^d$; see \cite{Kallenberg,LastPenrose17}.
In this note all random objects are defined over a fixed
probability space $(\Omega,\mathcal{F},\BP)$ with associated
expectation operator $\BE$.
Assume that $\Phi$ is {\em stationary}, that
is distributionally invariant under translations.
Assume also that $\Phi$ is {\em locally square integrable},
that is $\BE[\Phi(B)^2]<\infty$ for all compact $B\subset\R^d$.
Take a {\em convex body} $W$, that is a compact and 
convex subset of $\R^d$ and assume that $W$ has positive volume
$V_d(W)$. In many cases of interest one can define
an {\em asymptotic variance} by the limit
\begin{align}
\sigma^2:=\lim_{r\to\infty} \frac{\BV[\Phi(rW)]}{V_d(rW)},
\end{align}
where the cases $\sigma^2=0$ and $\sigma^2=\infty$ 
are allowed. This limit may depend on $W$; but we do
not include this dependence into our notation.  
Quite often the asymptotic variance $\sigma^2$ is
positive and finite. If, however, $\sigma^2=0$, then
$\Phi$ is said to be {\em hyperuniform}~\cite{TS03,T18}.
If $\sigma^2=\infty$, then $\Phi$ is said to be {\em hyperfluctuating}~\cite{T18}.
In recent years hyperuniform random measures (in particular point
processes) have attracted a great deal of attention.
The local behavior of such processes can very much resemble that of a
weakly correlated point process.
Only on a global scale a regular geometric pattern might become visible.
Large-scale density fluctuations remain anomalously suppressed similar
to a lattice; see \cite{TS03,T18,GL17a}.
The concept of hyperuniformity connects a broad range of areas of
research (in physics)~\cite{T18}, including
unique effective properties of heterogeneous materials,
Coulomb systems,
avian photoreceptor cells, 
self-organization, and isotropic photonic band gaps.

A point process $\Phi$ on $\R^d$ is said to be {\em number rigid}
if the number of points inside a given compact set is almost surely 
determined by the configuration of points outside \cite{PS14,B16}.
Examples of number rigid point processes include lattices independently 
perturbed by bounded random variables, Gibbs processes with certain 
long-range interactions \cite{DHLM18}, zeros of Gaussian entire 
functions \cite{GP17}, stable matchings from \cite{KLY20},  and some 
determinantal processes with a projection kernel \cite{GK15}.

It was proved in \cite{GL17b} that in one and two dimensions
a hyperuniform point process is number rigid, provided
that the truncated pair-correlation function is decaying sufficiently
fast.
Quite remarkably, it was shown in \cite{PS14} that in three and higher dimensions
a Gaussian independent perturbation of a lattice (which is hyperuniform)
is number rigid below a critical value of the variance
but not number rigid above.
It is believed \cite{GL17a} that a stationary number rigid point process
is hyperuniform.
In this note we show that this is not true.
In fact we give examples of stationary and ergodic (in 
fact mixing) random measures that are both hyperfluctuating and
rigid in a very strong sense.
The authors are not aware of any previously known rigid and ergodic 
process that is non-hyperuniform in dimensions $d\geq 2$ (if $W$ is the 
unit ball).
An example for $d=1$ has very recently been given in \cite{RLR20}.
In this paper we will prove that the point process 
resulting from intersecting Poisson hyperplanes 
has very strong rigidity properties.
This point process is hyperfluctuating \cite{HeinSS06} and, under an 
additional assumption on the directional distribution, mixing; see 
\cite[Theorem 10.5.3]{SW08} and Remark \ref{rmixing}.

\section{Poisson hyperplane processes}\label{sPoissonhyper}
 
In this section we collect a few basic properties of
Poisson hyperplane processes and the associated intersection processes.
Let $\BH^{d-1}$ denote the space of all
hyperplanes in $\R^d$. Any such hyperplane $H$ is of the form
\begin{align}\label{e12.31}
H_{u,s}:=\{y\in\R^d:\langle y,u\rangle =s\},
\end{align}
where $u$ is an element of the unit sphere $\BS^{d-1}$, $s\in\R$
and $\langle \cdot,\cdot\rangle$ denotes the Euclidean scalar product.
(Any hyperplane has two representations of this type.)
We can make $\BH^{d-1}$ a measurable space by introducing the
smallest $\sigma$-field containing  for each compact $K\subset \R^d$
the set
\begin{align}\label{ehit}
[K]:=\{H\in \BH^{d-1}:H\cap K\ne\emptyset\}.
\end{align}
In fact, $\BH^{d-1}\cup\{\emptyset\}$ can be shown to be a closed subset of the space of all
closed subsets of $\R^d$, equipped with the Fell topology. 
We refer to \cite{SW08} for more details on this topology
and related measurability issues; see also \cite[Appendix A3]{LastPenrose17}.

We consider a (stationary) {\em Poisson hyperplane process}, that is
a {\em Poisson process} $\eta$ on
$\BH^{d-1}$ whose intensity measure is given by
\begin{align}\label{elambda}
\lambda=\gamma\int_{\BS^{d-1}} \int_\R
\I\{H_{u,s}\in\cdot\}\,ds\,\BQ(du),
\end{align}
where $\gamma>0$ is an {\em intensity} parameter and $\BQ$ (the {\em directional distribution} of $\eta$)
is an even probability measure on $\BS^{d-1}$.
We assume that $\BQ$ is not concentrated on a great subsphere.
It would be helpful (even though not strictly necessary)
if the reader is familiar with basic point process and random measure
terminology; see e.g.\ \cite{LastPenrose17}.
For our purposes it is mostly enough to interpret $\eta$ 
as a random discrete subset of $\BH^{d-1}$. The number of points
(hyperplanes) in a measurable set $A\subset\BH^{d-1}$ is then given
by $|\eta\cap A|$ and has a Poisson distribution with parameter $\lambda(A)$.
Since $\lambda$ is invariant under translations
(we have for all $x\in\R^d$ that $\lambda(\cdot)=\lambda(\{H:H+x\in\cdot\})$),
the Poisson process $\eta$ is {\em stationary}, that is distributionally
invariant under translations. Furthermore we can derive from {\em Campbell's theorem}
(see e.g.\ \cite[Proposition 2.7]{LastPenrose17}) and \eqref{elambda} that 
\begin{align}\label{elf}
\BE[|\eta\cap [K]|]<\infty,\quad \text{$K\subset\R^d$ compact}.
\end{align}
As usual we assume (without loss of generality) that $|\eta(\omega)\cap [K]|<\infty$ 
for all $\omega\in\Omega$ and all compact $K\subset\R^d$.
More details on Poisson hyperplane processes
can be found in \cite[Section 4.4]{SW08}.

Let $m\in\{1,\ldots,d\}$. 
We define a random measure $\Phi_m$ on $\R^d$ by
\begin{align}\label{ePhim}
\Phi_m(B):=\frac{1}{m!}\;\sideset{}{^{\ne}}\sum_{H_1,\ldots,H_m \in \eta} 
\mathcal{H}^{d-m}(B\cap H_1\cap\cdots \cap H_m)
\end{align}
for Borel sets $B\subset\R^d$, 
where $\sum^{\neq}$ denotes summation over pairwise
distinct entries and where $\mathcal{H}^{d-m}$
is the Hausdorff measure of dimension $d-m$; see e.g.\ 
\cite[Appendix A.3]{LastPenrose17}. 
Using the arguments on p.\ 130 of \cite{SW08}
one can show that almost surely for all distinct $H_1,\ldots,H_m\in\eta$ the
intersection $H_1\cap\cdots\cap H_m$ is either empty or has dimension $d-m$.
Combining this with \eqref{elf}, we see that
the random measures $\Phi_1,\ldots,\Phi_m$ 
are almost surely {\em locally finite}, that is finite
on bounded Borel sets. 
The random variable $\Phi_m(B)$ is the
volume contents (in the appropriate dimension)
of all possible intersections of $d-m$ hyperplanes within $B$.

\begin{figure}[t]
  \centering
  \mbox{}\hfill%
  \includegraphics[width=0.38\textwidth]{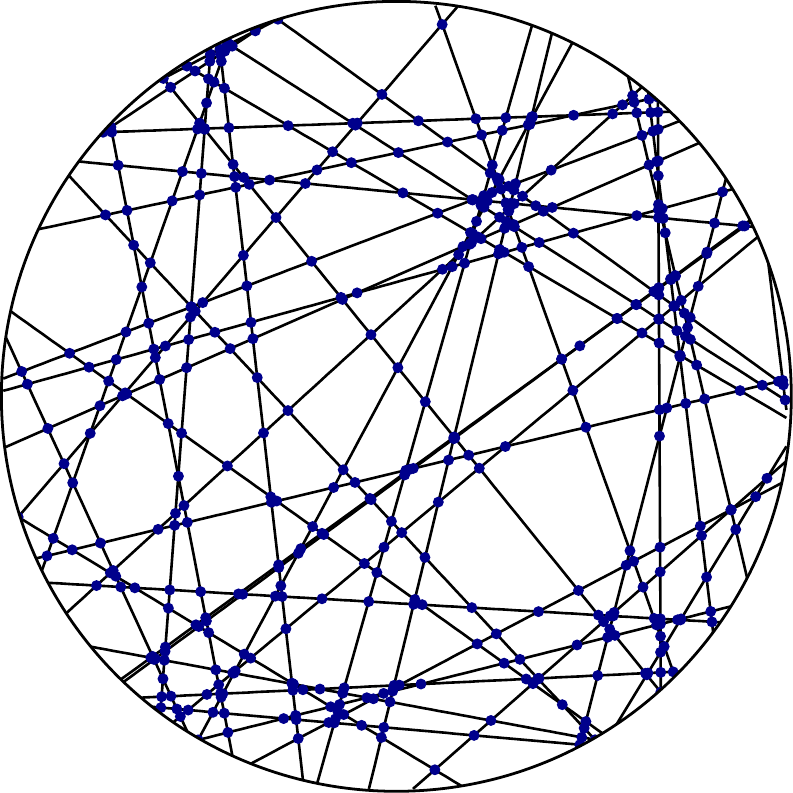}\hfill%
  \includegraphics[width=0.38\textwidth]{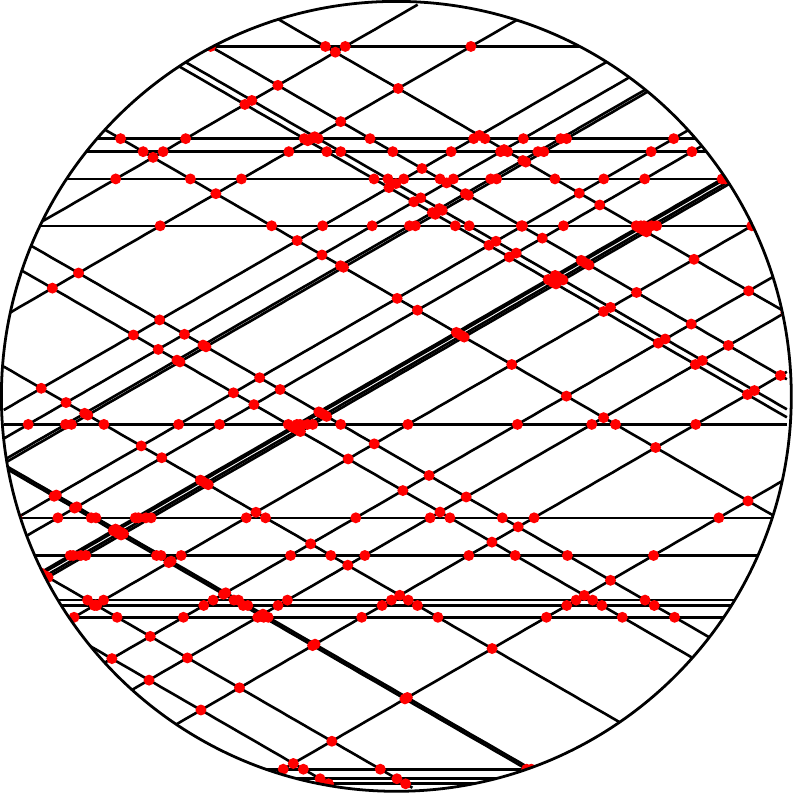}\hfill%
  \mbox{}\hfill%
  \caption{Samples of Poisson hyperplane processes $\eta$ (lines) and the 
  corresponding intersection processes $\Phi$ (solid circles) for two 
  directional distributions: isotropic (left) and only three directions 
  (right).}
  \label{fig:samples}
\end{figure}

It can be shown 
that (almost surely) the intersection of $d+1$ different hyperplanes from $\eta$
is empty. Therefore the random measure $\Phi_d$ is
almost surely a point process without multiplicities, so 
that $\Phi_d(B)$ is just the number of (intersection)
points $x\in B$ with $\{x\}=H_1\cap\cdots\cap H_d$ for
some $H_1,\ldots,H_d\in\eta$. It is convenient to define
a simple (and locally finite) point process
$\Phi$ as the set of all points   
$x\in\R^d$ with $\{x\}=H_1\cap\cdots\cap H_d$ 
for some $H_1,\ldots,H_d\in\eta$.
When (as it is common) interpreting $\Phi$ as a random
counting measure, we have that $\BP(\Phi=\Phi_d)=1$.
Figure~\ref{fig:samples} shows two samples of $\eta$ and $\Phi$.

Among other things, Theorem 4.4.8 in \cite{SW08} gives a formula
for the {\em intensity} $\gamma_m:=\BE[\Phi_m([0,1]^d)]$ of $\Phi_m$.
We only need to know that it is positive and finite.
In the remaining part of this section we recall
some second order properties of $\Phi_m$. (At first reading some details could
be skipped without too much loss.)
Let $A,B$ be bounded Borel subsets of $\R^d$.
Using the theory of U-statistics \cite[Section 12.3]{LastPenrose17}
it was shown in \cite{LPST14} that
\begin{align}\label{ecovariance}
\lim_{r\to\infty}r^{-(2d-1)}\BC[\Phi_m(rA),\Phi_m(rB)]=C_m(A,B),
\end{align}
where 
\begin{align}\label{cov}\notag
C_m(A,B)&:=\frac{1}{((m-1)!)^2}
\int\bigg(\int \mathcal{H}^{d-m}(A\cap H_1\cap \dots \cap H_m)\,
\lambda^{m-1}\big(d(H_2,\dots,H_m)\big)\bigg)\\
&\times\bigg(\int \mathcal{H}^{d-m}(B\cap H_1\cap H'_2\cap \dots \cap H'_m)\,
\lambda^{m-1}\big(d(H'_2,\dots, H'_m)\big)\bigg)\,\lambda(d H_1).
\end{align}
If $m=1$, this has to be read as
\begin{align*}
C_1(A,B)=
\int \mathcal{H}^{d-1}(A\cap H_1)\mathcal{H}^{d-1}(B\cap H_1)\,\lambda(d H_1).
\end{align*}
The asymptotic variance $C_m(A,A)$ was derived in \cite{HeinSS06}.
We note that $C_m(A,A)$ is finite (this is implied by the
form \eqref{elambda} of $\lambda$)
and that $C_m(A,A)=0$ iff  
$$
\int\mathcal{H}^{d-m}(A\cap H_1\cap \cdots \cap H_m)
\,\lambda^{m}(d(H_1,\dots,H_m))=0.
$$
Since $\BQ$ is not concentrated on a great subsphere, this
happens if and only if the Lebesgue measure of $A$ vanishes;
see the proof of \cite[Theorem 4.4.8]{SW08}.
Therefore we obtain from \eqref{ecovariance}
that the random measures $\Phi_1,\ldots,\Phi_d$ are hyperfluctuating
(if $d\ge 2$).
The results in \cite{HeinSS06,LPST14} show
that, for each finite collection $B_1,\ldots,B_n$ of bounded Borel sets,
the random vector 
$r^{-(d-1/2)}(\Phi_m(rB_1)-\BE[\Phi_m(rB_1)],\ldots,\Phi_m(rB_n)-\BE[\Phi_m(rB_n)])$
converges in distribution to a multivariate normal distribution.

It is worth noting that the asymptotic covariances \eqref{cov}
are non-negative. If $\eta$ is {\em isotropic} (meaning that
$\BQ$ is the uniform distribution on $\BS^{d-1}$),
there exist more detailed non-asymptotic second order results.
In this case \cite[p.\ 936]{HeinSS06} shows
the {\em pair correlation function} $\rho_2$ 
(see e.g.\ \cite[Section 8.2]{LastPenrose17})
of the intersection point process $\Phi=\Phi_d$ is given by
\begin{align}\label{epairc}
\rho_2(x)=1+\sum^d_{i=1}a_i \gamma^{-i}\|x\|^{-i},\quad x\in\R^d,\,x\ne 0,
\end{align}
where the coefficients $a_1,\ldots,a_d$ are strictly poisitive and do only 
depend on the dimension. Hence, as $\|x\|\to\infty$, $\rho_2(x)-1\to 0$
only at speed $\|x\|^{-1}$. In particular, the
{\em truncated} pair correlation function $\rho_2-1$ is not integrable
outside of any neighborhood of the origin.  
Using the well-known formula \cite[Exercise 8.9]{LastPenrose17}
\begin{align*}
\BV[\Phi(B)]=\gamma_dV_d(B)+\gamma_d^2\int V_d(B\cap (B+x))(\rho_2(x)-1)\,dx,
\end{align*}
(valid for all bounded Borel sets $B\subset\R^d$) and assuming that $B$
is convex, it is not too hard to confirm \eqref{ecovariance}
(using polar coordinates)  for $A=B$
and a certain positive constant $C_d(B,B)$. The value of 
this constant can be found in \cite{HeinSS06}.

\begin{remark}\label{rmixing}\rm Assume that $\BQ$ vanishes
on any great subsphere. Then
the random measures $\Phi_1,\ldots,\Phi_d$ have the
following {\em mixing} property. Let $i\in\{1,\ldots,d\}$. Then $\Phi_i$ can
be interpreted as a random element in a suitably space $\mathbf{M}$ of 
measures on $\R^d$ equipped with a suitable $\sigma$-field 
\cite{Kallenberg,LastPenrose17}.
Let $A,B$ be arbitrary measurable subsets of $\mathbf{M}$. Then
\begin{align*}
\lim_{\|x\|\to\infty}\BP(\Phi_i\in A,\theta_x\Phi_i\in B)=
\BP(\Phi_i\in A)\BP(\Phi_i\in B),
\end{align*}
where the random measure $\theta_x\Phi_i$ is defined by
$\theta_x\Phi_i(C):=\Phi_i(C+x)$ for Borel sets $C\subset\R^d$.
This is a straightforward consequence of \cite[Theorem 10.5.3]{SW08}
and the fact that $\Phi_i$ is derived from $\eta$ in a translation
invariant way. In particular $\Phi_i$ is {\em ergodic}, that
is $\BP(\Phi_i\in A)\in\{0,1\}$ for each translation invariant 
measurable set $A\subset\mathbf{M}$.
\end{remark}

\section{A reconstruction algorithm}\label{sreconstruction}

Let $\eta$ be a Poisson hyperplane process
as in Section \ref{sPoissonhyper}. 
Let $\Phi$ be the intersection point process
associated with $\eta$. (Recall from Section \ref{sPoissonhyper}  
that $\BP(\Phi=\Phi_d)=1$, where $\Phi_d$ is given by \eqref{ePhim}
for $m=d$.) Let $K\subset\R^d$ be
a non-empty convex and compact set. 
In this section we describe an algorithm which reconstructs 
$\eta\cap [K]$ (see Algorithm~\ref{algr} and Fig.~\ref{fig:algr})
by observing the points of $\Phi$ in a (random) bounded domain in the 
complement of $K$.
In the next section we shall show that this domain is exponentially 
small.

We say that $n\ge d$ points from $\R^d$
are in {\em general hyperplane position} if any $d$ of them
are affinely independent and span the same hyperplane.
The straightforward idea of the algorithm comes from
the following proposition of some independent interest.
We have not been able to find this result in the literature.

\begin{proposition}\label{p1} Almost surely the following is true.
Any distinct points $x_1,\ldots,x_{2d-1}\in\Phi$ in general
hyperplane position span a hyperplane $H\in\eta$.
\end{proposition}
\begin{proof} We start the proof with an auxiliary observation.
Let $m\in\N$ and let $f$ be a measurable function
on $(\BH^{d-1})^m$ taking values in the space of all non-empty closed subsets
of $\R^d$. (We equip this space
with the usual Fell--Matheron topology; see \cite{SW08}).
We assert that
\begin{align}\label{e9875}
\BP(\text{there exist distinct $H_1,\ldots,H_{m+1}\in\eta$ such
that $f(H_1,\ldots,H_m)\subset H_{m+1}$})=0.
\end{align}
Obviously the indicator function of the event in \eqref{e9875}
can be bounded by
\begin{align*}
X:=\sideset{}{^{\ne}}\sum_{H_1,\ldots,H_{m+1} \in \eta}\I\{f(H_1,\ldots,H_m)\subset H_{m+1}\}.
\end{align*}
If $\BE[X]=0$, then \eqref{e9875} follows.
By the multivariate Mecke formula \cite[Theorem 4.5]{LastPenrose17},
\begin{align*}
\BE[X]=\int \I\{f(H_1,\ldots,H_m)\subset H_{m+1}\}\,\lambda^{m+1}(d(H_1,\ldots,H_{m+1})).
\end{align*}
By Fubini's theorem it then enough to prove that
\begin{align}\label{e3.523}
\int \I\{F\subset H\}\,\lambda(dH)=0
\end{align}
for any non-empty closed set $F\subset\R^d$.
By monotonicity of integration it is sufficient 
to assume that $F=\{x\}$ for some $x\in\R^d$.
But then \eqref{e3.523} directly follows from \eqref{elambda}
and $\int \I\{\langle x,u\rangle=r\}\,dr=0$ for each $u\in \BS^{d-1}$.

We now turn to the main part of the proof.
Let $I_1,\ldots,I_{2d-1}\subset\N$ 
be distinct  with $|I_1|=\cdots=|I_{2d-1}|=d$. We shall refer to these sets
as {\em blocks} and to subsets of blocks as {\em subblocks}.
For convenience we assume that $I_1=[d]:=\{1,\ldots,d\}$.
Assume that $\cup^{2d-1}_{i=1} I_i=[n]$ for some $n\ge d$.  
Consider $(H_1,\ldots,H_{n})\in\eta^{n}$ with $H_i\ne H_j$ for $i\ne j$
and the following properties.
For each $i\in\{1,\ldots,2d-1\}$ we have that
$\cap_{j\in I_i} H_j$ consists of a single point $x_i$
and $x_1,\ldots,x_{2d-1}$ are in general hyperplane position.
Let $H$ be affine hull of $\{x_1,\ldots,x_{2d-1}\}$.
We will show that almost surely $H\in\{H_1,\ldots,H_{2d-1}\}$. 

Let us assume on the contrary that $H\notin\{H_1,\ldots,H_{2d-1}\}$.
Then each $k\in[n]$ (for instance $k=1$) belongs to at most
$d-1$ of the blocks. Indeed, by the general hyperplane
assumption we would otherwise have that $H_1=H$.
We will show that almost surely
\begin{align}\label{ecomb1}
\cap_{j\in I} H_j\subset H
\end{align}
for all subblocks $I$.

We prove \eqref{ecomb1}
by (descending) induction on the cardinality
$k$ of $I$. In the case $k=d$ \eqref{ecomb1} holds by definition
of $H$.
So assume that \eqref{ecomb1} holds for
all subblocks of cardinality $k\in\{2,\ldots,d\}$.
We need to show that it holds for each subblock
$I$ of cardinality $k-1$. For notational convenience
we take $I=[k-1]$. By induction hypothesis we have
that
\begin{align}\label{ee1}
H_1\cap\cdots\cap H_k\subset H.
\end{align}
Set $H':=H_1\cap\cdots\cap H_{k-1}\cap H$.
Since $H_1\cap\cdots\cap H_{k-1}\ne\emptyset$
we have (almost surely) that $\dim H_1\cap\cdots\cap H_{k-1}=(d-(k-1))$.
Since $\dim H=d-1$ we therefore obtain that $\dim H'\in\{d-k,d-(k-1)\}$.
Let us first assume that $\dim H'=d-k$.
By \eqref{ee1} (and since $H_1\cap\cdots\cap H_k\ne\emptyset$) 
we have that $\dim H_k\cap H'=d-k$. Therefore we obtain that
$H_k\cap H'=H'$, that is $H'\subset H_k$. Since $k$ is contained
in at most $d-1$ of the subblocks, for instance in $I_1,\ldots,I_{d-1}$,
the blocks $I_d,\ldots,I_{2d-1}$ still generate $H$,
that is $H=\aff\{x_d,\ldots,x_{2d-1}\}$. Therefore 
$H'$ is ``independent'' of $H_k$, contradicting $H'\subset H_k$.
More rigorously we can apply \eqref{e9875} to conclude that this
case can almost surely not occur.
Let us now assume that $\dim H'=d-(k-1)$. Then
$\dim H'=\dim H_1\cap\cdots\cap H_{k-1}$ and therefore
$H'=H_1\cap\cdots\cap H_{k-1}$.
This means that $H_1\cap\cdots\cap H_{k-1}\subset H$,
as required to finish the induction.

Using \eqref{ee1} for subblocks of size $1$, yields
that $H_k=H$ for each $k\in[n]$.
This contradiction finishes the proof of the lemma.
\end{proof}

\begin{remark}\rm In general it is not possible to reduce
the number $2d-1$ of points featuring in Proposition \ref{p1}. 
To see this, we may consider the case $d=3$ and
a directional distribution which is concentrated
on $\{e_1,e_2,e_3,-e_1,-e_2,-e_3\}$, where
$\{e_1,e_2,e_3\}$ is an orthonormal system.
In that case there exist infinitely many choices of four
intersection points in general hyperplane position
whose affine hull is not a hyperplane from $\eta$.
Indeed, the hyperplanes tessellate space into
cuboids and the four points can be chosen as endpoints 
of diametrically oposed edges of any cuboid.
\end{remark}

Our algorithm requires some notation. Let
$$
d(x,K):=\min\{\|y-x\|:y\in \R^d\}
$$
denote the Euclidean distance between $x\in\R^d$ and $K$
and let
\begin{align}
K_r:=\{x\in K^c:d(x,K)\le r\}
\end{align}
denote the {\em outer parallel set} of $K$ at distance $r\ge 0$.
Note that $K_0=\emptyset$.
Define random times $T_n$, $n\ge 1$, inductively by setting
\begin{align*}
T_{n+1}:=\min\{r>T_n:\Phi\cap (K_r\setminus K_{T_n})\ne\emptyset\},
\end{align*}
where $T_0:=0$. We form a (random) set  $\xi_n$ of hyperplanes 
as follows. A hyperplane $H$ belongs to $\xi_n$ if it
does not intersect $K$ and if it contains $2d-1$ different points
from $\Phi\cap K_{T_n}$ in general hyperplane position.
By Proposition \ref{p1} we have almost surely that $\xi_n\subset\eta$.
For a hyperplane $H$ with $H\cap K=\emptyset$ 
we let $H(K)$ denote the half-space bounded by $H$ with $K\subset H(K)$.

\begin{figure}[t]
  \centering
  \includegraphics[width=\textwidth]{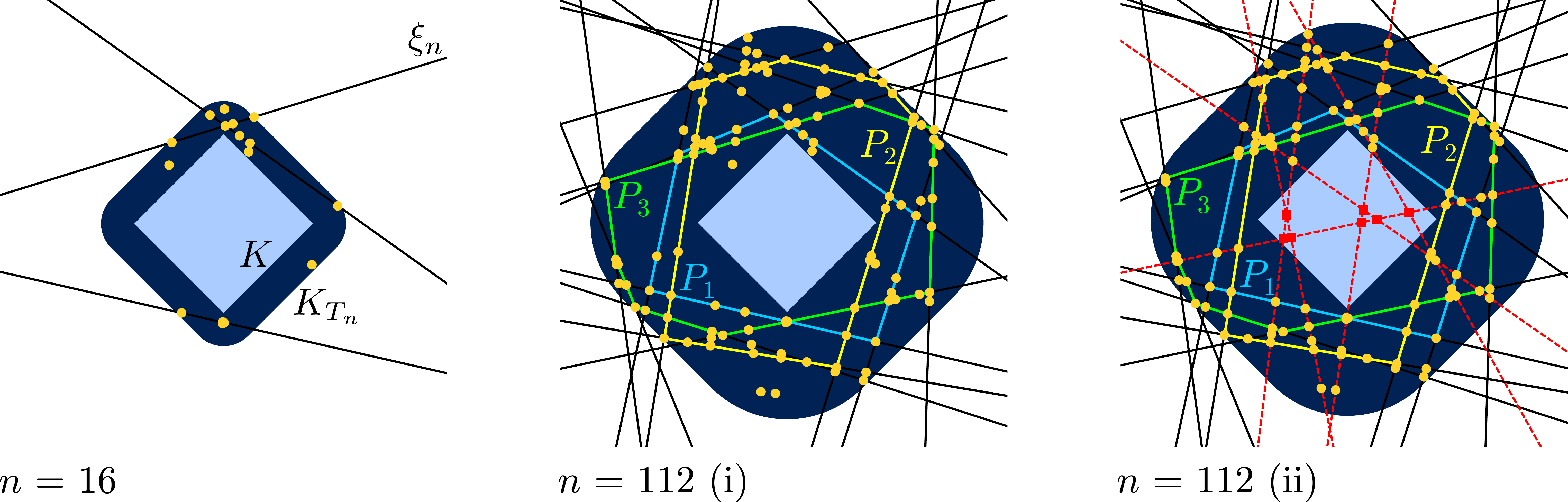}
  \caption{Reconstruction algorithm of $\eta\cap [K]$: Given a convex 
  domain $K$, the algorithm recursively scans the points in $\Phi\cap 
  K_{T_n}$ (solid circles).
  At step $n=16$, three hyperplanes are reconstructed (left).
  At step $n=112$, three polygons $P_1$, $P_2$, $P_3$ are reconstructed 
  within the outer parallel set $K_{T_n}$ (center).
  Hence, all hyperplanes in $\eta\cap [K]$ (dashed lines) 
  can be reconstructed (right).}
  \label{fig:algr}
\end{figure}

\begin{algorithm}\label{algr}\rm
The algorithm iterates over the random times $T_n$, $n\ge 1$,
(recursively) scanning the points in $\Phi\cap K_{T_n}$.
If the algorithm stops at time $T_n$, 
then it returns a set $\chi_n$ of hyperplanes
that will be proved to coincide (almost surely) with $\eta\cap [K]$.
Stage $n$ of the algorithm is defined as follows 
(cf.~Fig.~\ref{fig:algr}):

\begin{enumerate}
\item[(i)] Determine $\xi_n$ and check whether 
there are integers $k_1,\ldots,k_{2d-1}$ and
distinct $H_{i,j}\in\xi_n$ ($i\in[k_j]$, $j\in[2d-1]$)
such that the boundary of
\begin{align*}
  P_j:=\bigcap^{k_j}_{i=1}H_{i,j}(K) 
\end{align*}
is contained in $K_{T_n}$ for each $j\in[2d-1]$.
If such hyperplanes do not exist, the algorithm continues with stage $n+1$. 
If they do exist, the algorithm continues with step (ii) and stops after it.
\item[(ii)] Find all collections of $2d-1$ points 
in $\Phi\cap K_{T_n}$ in general hyperplane position such
that the generated hyperplane intersects $K$.
If there are such points, $\chi_n$ is the set of all those hyperplanes.
If there are no such points, then $\chi_n:=\emptyset$.
\end{enumerate}

Let $T:=T_n$ if the algorithm stops at stage $n$.
We set $T:=\infty$ if it never stops. We can interpret $T$
as the running time of the algorithm in continuous time.
In the next section we will not only show that $T$ is (almost surely) 
finite but does also have exponential moments.
Here we wish to assure ourselves of the essentially geometric fact that 
the algorithm indeed determines $\eta\cap [K]$.
\end{algorithm}

\begin{proposition}\label{palg} On the event $\{T<\infty\}$ we have almost surely
that $\chi_T=\eta\cap [K]$.
\end{proposition}
\begin{proof} Assume that the algorithm stops at 
at stage $n$  
and let $P_1,\ldots,P_{2d-1}$ be as in step (i) of the algorithm.
These are bounded polytopes which contain
$K$ in their interior and which are made up of different hyperplanes from $\eta$.
Assume that $H\in\eta$ intersects $K$. Then $H$ intersects
for each $i\in[2d-1]$ the boundary of the polytope $P_i$,
and in fact, at least one of its edges. Therefore there exist distinct
hyperplanes $H_1,\ldots,H_{(2d-1)(d-1)}\in\eta\setminus\{H\}$ such 
that 
\begin{align*}
  H\cap \bigcap_{j\in I_i} H_j\ne \emptyset,\quad i\in[2d-1],
\end{align*}
where $I_i:=\{(i-1)(d-1)+1,\ldots,i(d-1)\}$.
Almost surely each of these intersections consists of only one point $x_i$, say.
We assert that these points are in general hyperplane position.
If they are not, then $d$ among those points, $x_1,\ldots,x_d$ say, 
are affinely dependent.
Then one of those points, $x_d$ say, must lie in $\aff\{x_1,\ldots,x_{d-1}\}$.
Therefore we need to show that
the probability of finding distinct $H_0,\ldots,H_{d(d-1)}\in\eta$ such that
$\{x_i\}:=H_0\cap \cap_{j\in I_i}H_j$ is a singleton for
each $i\in[d]$ and  
\begin{align*}
x_d\in\aff\{x_1,\ldots,x_{d-1}\}
\end{align*}
is zero. Similarly as in the proof of \eqref{e9875}
this probability can be bounded by
\begin{align*}
\iiint &\I\{|H\cap\cap_{j\in I_1}H_j|=\cdots=|H\cap\cap_{j\in I_d}H_j|=1\}\\
&\I\{H\cap \cap_{j\in I_d}H_j\subset
\aff \big(H\cap \cap_{j\in I_1}H_j,\ldots,H\cap\cap_{j\in I_{d-1}}H_j)\}\\
&\qquad
\lambda(dH)\,\lambda^{d(d-1)}(d(H_1,\ldots,H_{d(d-1)})).
\end{align*}
Therefore it is enough to show that for $\lambda$-a.e.\ $H\in\BH^{d-1}$
and each affine space $E\subset H$ of dimension at most $d-2$ 
\begin{align}\label{elambdaH}
\int \I\{|\cap_{j\in I_d}H_j|=1, \cap_{j\in I_d}H_j\subset E\}
\,\lambda^{d-1}_H(d(H_1,\ldots,H_{d-1}))=0,
\end{align}
where $\lambda_H$ is the measure on the space 
of all affine subspaces of $H$ given by
\begin{align*}
\lambda_H:=\int\I\{H'\cap H\in\cdot\}\,\lambda(dH').
\end{align*}
For $\lambda$-a.e.\ $H$, the measure $\lambda_H$ is concentrated
on the $(d-2)$-dimensional subspaces of $H$ and invariant
under translations in $H$. In fact, $\lambda_H$ is the intensity
measure of the Poisson process $\eta_H:=\{H'\cap H: H'\in\eta\}$.
Up to a constant multiple,
$$
B\mapsto \int \I\{B\cap H_1\cap\cdots \cap H_{d-1}\ne \emptyset\}\,\lambda^{d-1}_H(d(H_1,\ldots,H_{d-1}))
$$
is (as a function of the Borel set $B\subset H$)
the intensity measure of  the intersection process associated
with $\eta_H$
(see \cite[p.\ 135]{SW08}) and therefore
proportional to Lebesgue measure on $H$. 
(It can also be checked more directly, that
this function is a locally finite translation invariant measure.)
Hence \eqref{elambdaH} follows.
\end{proof}

\begin{remark}\rm Assume that
  that the directional distribution $\BQ$ is absolutely continuous
  with respect to Lebesgue measure on $\BH^{d-1}$. Then the
  algorithm can be considerably simplified. In step (i) it is enough  
  two find just two polytopes $P_1,P_2$ made up of distinct hyperplanes
  in $\xi_n$. Any hyperplane $H$ from $\eta$ that intersects $K$,
  intersects the boundary of the polytope $P_1$ in $d$ affinely independent points from $\Phi$
  and the same applies to $P_2$ (even without further assumptions of $\BQ$). With some
  efforts it can be shown that the resulting $2d$ intersection points are almost
  surely in general hyperplane position. The forthcoming Theorem \ref{tstrongrigid1}
  remains valid. We do not go into the technical details.
\end{remark}

\begin{remark}\rm 
  The reconstruction algorithm~\ref{algr} is not optimized for computational
  efficiency.  For instance, the algorithm could already be stopped
whenever there exists just one polytope which contains $K$ in its
interior but no points of $\Phi$ in the relative interior of its edges.
(In this case $\eta\cap[K]=\emptyset$.)
Moreover, it is not necessary that the 
boundaries of the polytopes $P_j$  are
completely contained in $K_{T_n}$. 
It would suffice to find polyhedral sets with sufficiently large parts
of their boundaries contained in $K_{T_n}$.
\end{remark}

\section{Strong rigidity}\label{srigid}

In this section we shall exploit
the algorithm from Section \ref{sreconstruction}
to show that the intersection processes
$\Phi_1,\ldots,\Phi_m$ associated with a
Poisson hyperplane process have very strong rigidity properties.

We start with giving a few definitions. Let $\Psi$ be a random
measure on $\R^d$ (for instance one of the $\Phi_1,\ldots,\Phi_m$).
For a Borel set $B\subset\R^d$ we denote by $\Psi_B:=\Psi(\cdot\cap B)$ the restriction
of $\Psi$ to $B$. 
A mapping
$Z$ from $\Omega$ into the space of non-empty closed subsets of $\R^d$
is called {\em $\Psi$-stopping set}
if $\{Z\subset F\}:=\{\omega\in\Omega:Z(\omega)\subset F\}$ is for each closed set $F\subset \R^d$
an element of the $\sigma$-field $\sigma(\Psi_F)$ generated by $\Psi_F$.
(In particular $Z$ is then a {\em random closed set}  \cite{Molchanov05}.)
By $\Psi_Z$ we understand the restriction of $\Psi$ to $Z$
(that is the random measure $\omega\mapsto \Psi(\omega)_{Z(\omega)}$.)
If $Z$ is a $\Psi$-stopping set, then we say that $\eta\cap [K]$ is 
{\em almost surely determined by $\Psi_Z$},
if there exists a measurable mapping $f$ (with suitable domain)
such that $\eta\cap [K]=f(\Psi_Z)$ holds almost surely.

The following result shows that $\eta\cap[K]$ is almost surely determined
by a $\Phi$-stopping set $Z\subset K^c$ of exponentially small size.
Here we quantify the size of a closed set $F\subset\R^d$ by
the radius $R(F)$ of the smallest ball centred at the origin and containing $F$.
(If $F$ is not bounded then we set $R(F):=\infty$.)

\begin{theorem}\label{tstrongrigid1} Let $K\subset\R^d$ be convex and compact. Then
there exists a $\Phi$-stopping set $Z$ with $Z\subset K^c$ and such
that $\eta\cap [K]$ is almost surely determined by $\Phi\cap Z$. 
Moreover, there exist constants $c_1,c_2>0$ such
that
\begin{align}\label{eexp}
\BP(R(Z)>s)\le c_1e^{-c_2s},\quad s\ge 1.
\end{align}
\end{theorem}
\begin{proof} We consider the algorithm from Section \ref{sreconstruction}
with running time $T$, defined after Algorithm~\ref{algr}.
We assert that $Z:=K_{T}$ has all desired properties,
where $K_\infty:=K^c$. The inclusion $Z\subset K^c$ is a direct consequence
of the definitions. The stopping set property can be considered as pretty much
obvious. The reader might wish to skip the following technical argument.
Define $\bN$ as the set of all locally finite subsets of $\R^d$.
The algorithm from Section \ref{sreconstruction} can be used (in an obvious way)
to define a measurable mapping $\tilde{Z}$ from $\bN$ (equipped with
the standard $\sigma$-field) to the space of all closed
subsets of $\R^d$ such that $Z=\tilde Z(\Phi)$.
We need to show that $\tilde Z$ is a stopping set, that is
$\{\mu:\tilde Z(\mu)\subset F\}$ is for all closed sets $F\subset\R^d$
an element of the $\sigma$-field generated by the mapping
$\mu\mapsto \mu\cap F$ from $\bN$ to $\bN$. To prove this we use
\cite[Proposition A.1]{BaumstarkLast09}. According to this proposition
it is sufficient to show that 
$\tilde{Z}((\psi\cap \tilde{Z}(\psi))\cup\varphi)=\tilde{Z}(\psi)$
for all $\psi,\varphi\in\bN$ with $\varphi\subset\tilde{Z}(\psi)^c$.
But this follows from the definition of the algorithm. Indeed, suppose
that $\psi\in\bN$ is a realization of the intersection process
and that the algorithm stops at time $t$. Restricting $\psi$ to $K_t$ and
then adding a configuration $\varphi$ in the complement of $K_t$ does not change the
running time $t$. 

We show \eqref{eexp} by modifying 
the idea of the proof of Lemma 1 in \cite{Schneider19}.
Since $\BQ$ is not concentrated on a great subsphere there exist
linearly independent vectors $e_1,\ldots,e_{d}\in\R^d$ in the support of $\BQ$.
Since $\BQ$ is even, the vectors $e_{d+1}:=-e_1,\ldots,e_{2d}:=-e_{d}$ are
also in the support of $\BQ$.
We can then find a (large) constant $b>0$ and
(small) pairwise disjoint 
closed neighborhoods $U_i$  of $e_i$, $i\in\{1,\ldots,2d\}$, such that 
$U_{d+i}=\{-u:u\in U_i\}$ and each intersection
\begin{align*}
P=\bigcap^{2d}_{i=1} H^-(u_i,1)
\end{align*}
with $u_i\in U_i$, $i\in\{1,\ldots,2d\}$,
is a polytope with $R(P)\le b$. 
Here we write, for given $u\in\R^d$ and $s\in\R$,
$H^-(u,s):=\{y\in\R^d:\langle y,u\rangle \le s\}$.
Let $t\ge 0$. From linearity of the scalar product we then obtain
that
\begin{align}\label{eRbound}
R\bigg(\bigcap^{2d}_{i=1} H^-(u_i,t_i)\bigg)\le b(R(K)+t),
\end{align}
whenever $R(K)\le t_i\le R(K)+t$ and $u_i\in U_i$ for $i\in\{1,\ldots,2d\}$.

We need a straightforward analytic fact.
Since the determinant is a continuous function we can assume
that there exists $a>0$ such that
\begin{align}\label{edet}
|\det(u_1,\ldots,u_d)|\ge a,\quad 
(u_1,\ldots,u_d)\in {U}_1\times\cdots\times {U}_d.
\end{align}
For $i\in\{1,\ldots,d\}$
let $u_i\in U_i\cup U_{d+i}$ and $s_i\in\R$. Then
$H_{u_1,s_1}\cap\cdots\cap H_{u_d,s_d}$ consists of a single
point $x$ (by \eqref{edet} and $U_{d+i}=-U_i$),
whose Euclidean norm can be bounded as
\begin{align}\label{edist}
  \|x\|\le b'\max\{|s_i|:i=1,\ldots,d\},
\end{align}
where $b'>0$ is a constant that depends only on the dimension
and the (fixed) sets $U_1,\ldots,U_d$.
To see this we note that $x$ (now interpreted as a column vector)
is the unique solution of the linear equation $Ax=s$,
where $A$ is the matrix with rows $u_1,\ldots,u_d$ and $s$ is
the column vector with entries $s_1,\ldots,s_d$.
By \eqref{edet} we have that $x=A^{-1}s$. It is well-known that
\begin{align*}
\|x\|_\infty\le \|A^{-1}\|_\infty \|s\|_\infty,
\end{align*}
where $\|x\|_{\infty}:=\max\{|x_i|:i=1,\ldots,d\}$
and $\|A^{-1}\|_\infty$ is the maximum absolute row sum of $A^{-1}$.
In view of the explicit expression of $A^{-1}$ in terms
of $\det(A)^{-1}$ and the minors of $A$ and
the minimum principle for continuous functions we have
that $\|A^{-1}\|_\infty$ is bounded from above by a positive
constant. (Recall that $u_1,\ldots,u_d$ are unit vectors.)
Since $\|x\|\le c \|x\|_{\infty}$ for some $c>0$ we obtain \eqref{edist}.

For notational simplicity we now assume that $K$ is a ball
with radius $R$ centred at the origin. In fact, in view of the assertion
this is no restriction of generality.
Consider the following sets of hyperplanes:
\begin{align*}
A_i(t):=\{H(u,s):u\in U_i,R< s\le R+t\},\quad i\in[2d].
\end{align*}
We assert the event inclusion
\begin{align}\label{e3.44}
\bigcap^{2d}_{i=1}\{|\eta\cap A_i(t)|\ge 2d-1\}\subset \{R(Z)\le b''(R+t)\},\quad \BP\text{-a.s.}, 
\end{align}
where $b'':=\max\{b,b'\}$ with $b'$ as in \eqref{edist}.

To show \eqref{e3.44}, we assume that $|\eta\cap A_i(t)|\ge 2d-1$ for each $i\in[2d]$.
Then we can find distinct hyperplanes $H_{i,j}\in\eta$ 
($i\in[2d]$, $j\in[2d-1]$) not intersecting $K$
such that the polytopes
\begin{align*}
P_j:=\bigcap^{2d}_{i=1} H_{i,j}(K),\quad j\in[2d-1],
\end{align*}
contain $K$ in their interior and satisfy $R(P_j)\le b(R+t)$;
see  \eqref{eRbound}. Next we show that
each $H_{i,j}$ is in $\xi_n$ as soon as
$T_n\ge b''(R+t)-R$. (Then our algorithm has
identified these hyperplanes by time $T_n$.)
Take $H_{1,1}$, for instance. Define $x_1,\ldots,x_{2d-1}\in\Phi$ by
$\{x_j\}:=H_{1,1}\cap\cap^d_{i=2}H_{i,j}$. It can then be shown
as in the proof of Proposition \ref{palg} that these
points are in general hyperplane position.
Therefore we obtain from \eqref{edist}
and the definition of $A_1(t)$ that $\|x_1\|\le b'(R+t)$
and in fact $\|x_j\|\le b'(R+t)$ for each $j\in[2d-1]$.
Therefore $H_{i,j}\in \xi_n$, provided that $T_n\ge b'(R+t)-R$. 
We have already seen that $R(P_j)\le b(R+t)$, so that the boundary of
$P_j$ is contained in $K_{T_n}$ if $T_n\ge b(R+t)-R$. (Note that $K_{T_n}$
is a spherical shell with outer radius $R+T_n$ centred at the origin.)
Altogether we obtain that $T\le b''(R+t)-R$ and hence
$R(Z)=R(K_T)\le b''(R+t)$, proving \eqref{e3.44}.

Having established \eqref{e3.44} we next note that
\begin{align*}
\BP(R(Z)> b''(R+t))&\le \BP\bigg(\bigcup^{2d}_{i=1}\{|\eta\cap A_i(t)|\le 2d-2\}\bigg)\\
&\le \sum^{2d}_{i=1}\BP(|\eta\cap A_i(t)|\le 2d-2)\\
&=\sum^{2d}_{i=1}\exp[-\lambda(A_i(t))]\sum^{2d-2}_{j=0}\frac{\lambda(A_i(t))^j}{j!},
\end{align*}
where we have used the defining properties of a Poisson process
to obtain the final equality. By \eqref{elambda} we have that
\begin{align}
\lambda(A_i(t))=\gamma t\,\BQ(U_i).
\end{align}
Setting $a:=\min\{\BQ(U_i):i\in[2d]\}$ 
and using that $\BQ(U_i)\le 1$ for each $i\in[2d]$, we obtain that
\begin{align*}
\BP(R(Z)> b(R+t))\le 2de^{-\gamma a t}\sum^{2d-2}_{j=0}\frac{\gamma^j}{j!}t^j.
\end{align*}
This implies \eqref{eexp} for suitably chosen $c_1,c_2$.
\end{proof}

\begin{remark}\rm The stopping set $Z$ in Theorem \ref{tstrongrigid1}
depends measurably on $\Phi\cap K^c$ (is a measurable function of $\Phi\cap K^c$).
This follows from the definition of the algorithm, but
also from the following argument, which applies to general stopping sets
$Z$ with the property $Z\subset K^c$.
By standard properties of random
closed sets it suffices to check for each compact $F\subset \R^d$ that 
$\{Z\cap F=\emptyset\}\in \sigma(\eta\cap K^c)$. Since $Z\subset K^c$ we have that
$\{Z\cap F=\emptyset\}=\{Z\cap (F\cup K)=\emptyset\}$.
Since $F\cup K$ is compact, there is a decreasing sequence $(U_n)_{n\ge 1}$
of open sets with intersection $F\cup K$ and such that
\begin{align*}
\{Z\cap (F\cup K)=\emptyset\}=\bigcup^\infty_{n=1}\{Z\cap U_n=\emptyset\}
=\bigcup^\infty_{n=1}\{Z\subset U^c_n\}.
\end{align*}
Since $Z$ is a $\Phi$-stopping set we have that the above right-hand side
is contained in $\cup^\infty_{n=1}\sigma(\Phi_{U_n^c})\subset\sigma(\Phi_{K^c})$,
as asserted.
\end{remark}

Theorem \ref{tstrongrigid1} implies the announced strong
rigidity properties of the intersection processes.

\begin{theorem}\label{tstrongrigid2} Let $m\in\{1,\ldots,d\}$
and let $B\subset\R^d$ be a bounded Borel set.
Then there exists a $\Phi_m$-stopping set $Z$ with $Z\subset B^c$ and such
that $(\Phi_m)_B$ is almost surely determined by $(\Phi_m)_Z$. 
Moreover, there exist constants $c_1,c_2>0$ such
that \eqref{eexp} holds.
\end{theorem}
\begin{proof}
Choose a convex and compact set $K\subset\R^d$ with $B\subset K$.  
Clearly, if the assertion holds in the case $B=K$, then we
obtain it for all $B\subset K$. Hence we can assume that $B=K$.
Let $Z$ be as in Theorem \ref{tstrongrigid1}. Since $\Phi_F$
is for each closed $F\subset \R^d$ a measurable function 
of $(\Phi_m)_F$ it follows that $Z$ is a $\Phi_m$-stopping set.
Moreover, $(\Phi_m)_K$ is a (measurable) function of
$\eta\cap [K]$. Hence Theorem \ref{tstrongrigid1} implies the
assertions. 
\end{proof}

The rigidity property in Theorem \ref{tstrongrigid2}
is considerably stronger than the {\em strong rigidity} studied 
in \cite{GL18}. The random measure $(\Phi_m)_B$ is 
not only determined by $(\Phi_m)_{B^c}$, but already
by $(\Phi_m)_{Z}$ for an exponentially small stopping
set $Z\subset B^c$.

\section{Hyperfluctuating Cox processes and thinnings}\label{sCox}

The strong rigidity property of the random measures $\Phi_1,\ldots,\Phi_m$
can easily be destroyed by additional randomization. For example
we may consider, for $m\in\{1,\ldots,d\}$, a {\em Cox process}
$\Psi_m$ directed by $\Phi_m$ \cite[Chapter 13]{LastPenrose17}. 
This means that the conditional distribution
of $\Psi_m$ given $\Phi_m$ is that of a Poisson process
with intensity measure $\Phi_m$. For $m=d$ this point process
can be interpreted as a multiset (or a random measure). 
Each point of $\Phi_m$ gets (independently of the other points)
a random multiplicity having a Poisson distribution of mean 1.
Let $B\subset\R^d$ be a bounded Borel set. Then the well-known
conditional variance formula (together with the stationarity of $\Phi_m$)
implies that
\begin{align*}
\BV[\Psi_m(B)]=\gamma_mV_d(B)+\BV[\Phi_m(B)],
\end{align*}
where $\gamma_m$ is the intensity of $\Phi_m$; see 
\cite[Proposition 13.6]{LastPenrose17}.
By \eqref{ecovariance}, $\Psi_m(B)$ has the same variance asymptotics
as $\Phi_m(B)$. In particular, $\Psi_m$ is (for $d\ge 2$) hyperfluctuating.
However, $\Psi_m$ is not rigid. For example, given a Borel set
$B$ with positive volume, $\Psi_m(B)$ is not determined
by the restriction of $\Psi_m$ to the complement of $B$.

In the case of the intersection point process $\Phi$
there is an even simpler way of randomizing, namely
to form a {\em $p$-thinning} $\Phi_p$ of $\Phi$
for some $p\in(0,1)$. Formally, given $\Phi$, the points of $\Phi$
are taken independently of each other as points of $\Phi_p$ with probability $p$
\cite[Section 5.3]{LastPenrose17}. 
This point process is not rigid.
A simple calculation (using the conditional variance formula for instance)
shows that 
\begin{align*}
\BV[\Phi_p(B)]=p^2\BV[\Phi(B)]+p(1-p)\BE[\Phi(B)],
\end{align*}
so that $\Phi_p$ inherits the variance asymptotics from $\Phi$.
It also not hard to see that the pair correlation function
of $\Phi_p$ is the same as that of $\Phi$ and hence 
given by the slowly decaying function \eqref{epairc}.

\section{Concluding remarks}

We have shown that the intersection point process $\Phi$ associated with 
a stationary Poisson hyperplane process is rigid in a very strong sense.
This holds for any directional distribution which is not concentrated on 
a great subsphere.
(Our arguments suggest that this might be true for more general 
stationary and mixing hyperplane processes with absolute continuous 
factorial moment measures.) On the other hand, $\Phi$ is 
hyperfluctating.
Hence hyperuniformity is not necessary for rigidity as (weakly) 
conjectured in \cite{GL17b}.
However, we completely agree with the authors of \cite{GL17b} that the 
precise relationships between rigidity and hyperuniformity constitute an 
interesting intriguing problem.
We believe in the existence of generic point process assumptions that 
need to be added to rigidity to conclude hyperuniformity.
Preferably these assumptions should be as minimal as possible.

\bigskip
\noindent 
\textbf{Acknowledgements:}
This research was supported in part by the Princeton University 
Innovation Fund for New Ideas in the Natural Sciences.
The authors wish to thank Daniel Hug and
Salvatore Torquato for valuable discussions  of
some aspects of our paper.

\end{document}